\documentclass[11pt]{amsart}

\usepackage{epigamath}

\usepackage[english]{babel}

\numberwithin{equation}{section}

\usepackage[shortlabels]{enumitem}

\usepackage{mathrsfs}
\usepackage{amssymb, amsopn, amsmath, array, amsfonts, float}

\usepackage{tikz, tikz-cd}
\usetikzlibrary{decorations.markings}
\usepackage[capitalise]{cleveref}
\usepackage{multicol}
\usepackage{booktabs}

\tikzset{degil/.style={
            decoration={markings,
            mark= at position 0.5 with {
                  \node[transform shape] (tempnode) {$\backslash$};
                  }
              },
              postaction={decorate}
}
}

\definecolor{cverbbg}{gray}{0.93}

\allowdisplaybreaks

\newtheorem{Theoremx}{Theorem}

\newtheorem{theorem}{Theorem}[section]

\newtheorem{proposition}[theorem]{Proposition}

\theoremstyle{definition}

\theoremstyle{remark}
\newtheorem{example}[theorem]{Example}

\newcommand{\tpar}{\tau_{\operatorname{par}}}
\newcommand{\im}{\operatorname{Im}}

\newcommand{\Jac}{\operatorname{Jac}}

\newcommand{\Spec}{\operatorname{Spec}}

\newcommand{\Ext}{\operatorname{Ext}}

\newcommand{\NN}{\mathbb{N}}

\newcommand{\N}{\mathbb{N}}

\newcommand{\Z}{\mathbb{Z}}
\newcommand{\F}{\mathbb{F}}

\newcommand{\PP}{\mathbb{P}}

\newcommand{\lo}{\longrightarrow}

\newcommand{\fm}{\mathfrak{m}}
\newcommand{\fp}{\mathfrak{p}}

\newcommand{\fn}{\mathfrak{n}}
\newcommand{\fM}{\mathfrak{M}}

\newcommand{\cF}{\mathcal{F}}

\newcommand{\cP}{\mathcal{P}}

\setlist[enumerate,1]{label={\rm(\alph*)}, ref={\rm\alph*}}

\newlist{roenum}{enumerate}{2}
\setlist[roenum,1]{label={\rm(\roman*)}, ref={\rm\roman*}}

\crefname{Theoremx}{Theorem}{Theorems}
\crefname{figure}{Figure}{Figures}

\newcommand{\supth}[1]{\ensuremath{#1^{\mathrm{th}}}}

\newcommand{\Nup}{\mathrm{N}}

\newcommand{\longhookrightarrow}{\lhook\joinrel\longrightarrow}

\EpigaVolumeYear{10}{2026} \EpigaArticleNr{21} \ReceivedOn{October 10, 2025}
\InFinalFormOn{March 10, 2026}
\AcceptedOn{May 21, 2026}

\title{$\boldsymbol{F}$-injectivity does not imply $\boldsymbol{F}$-fullness in normal domains}
\titlemark{$\boldsymbol{F}$-injectivity does not imply $\boldsymbol{F}$-fullness in normal domains}

\author{Alessandro De Stefani}
\address{Dipartimento di Matematica, Universit\`{a} di Genova, Via Dodecaneso 35, 16146 Genova, Italy}
\email{alessandro.destefani@unige.it}

\author{Thomas Polstra}
\address{Department of Mathematics, University of Alabama, Tuscaloosa, AL 35487,  USA}
\email{tmpolstra@ua.edu}

\author{Austyn Simpson}
\address{Department of Mathematics, Bates College, 3 Andrews Rd, Lewiston, ME 04240 USA}
\email{asimpson2@bates.edu}

\authormark{A.~De Stefani, T.~Polstra and A.~Simpson}

\AbstractInEnglish{We construct examples of Noetherian three-dimensional local geometrically normal domains of prime characteristic that are $F$-injective but not $F$-full. Along the way, we find examples of two-dimensional local geometrically normal domains that are $F$-injective but not $F$-anti-nilpotent. A crucial theme of our constructions is the behavior of $F$-injectivity along a purely inseparable finite base change.}

\MSCclass{13A35, 13D45}

\KeyWords{Frobenius, $F$-singularities, local cohomology, normal domains}

\acknowledgement{De Stefani was partially supported by the MIUR Excellence Department Project CUP D33C23001110001, PRIN 2022 Project 2022K48YYP, and by INdAM-GNSAGA.}

\begin{document}



\maketitle

\begin{prelims}

\DisplayAbstractInEnglish

\bigskip

\DisplayKeyWords

\medskip

\DisplayMSCclass

\end{prelims}


\newpage

\setcounter{tocdepth}{1}

\tableofcontents


\section{Introduction}
Throughout this introduction, let $(R,\fm)$ be a $d$-dimensional reduced local ring that is essentially of finite type over an $F$-finite field $k$ of prime characteristic $p>0$. Recall that $R$ is \emph{$F$-injective} provided that the maps $F\colon H^i_\fm(R)\to H^i_\fm(R)$ induced by the Frobenius endomorphism are injective for all $i\leq d$. Introduced in \cite{Fed83}, $F$-injective rings are among the least restrictive (yet still mild) singularities detected by the Frobenius map, and they have long been suspected to be the correct prime-characteristic counterpart of the Du Bois singularities of the complex minimal model program; see \cite{Sch09}. However, they often lack certain structural properties enjoyed by the other $F$-singularities. Most starkly, unlike $F$-regular and $F$-rational singularities, $F$-injective rings need not be normal or Cohen--Macaulay. The relationship between $F$-injectivity and $F$-purity is more subtle; recall that $R$ is \emph{$F$-pure} under our current assumptions if the inclusion $R\hookrightarrow R^{1/p}$ admits an $R$-linear left inverse.

$F$-pure rings are easily seen to be $F$-injective, and both classes satisfy weak normality. However, the former exhibit certain mild behavior that is either false or unknown for the latter. Here are some examples: 
\begin{enumerate}[label=(\Roman*),ref=\Roman*]
    \item (\emph{WN1}\,) If $R$ is $F$-pure, the normalization map $R\to R^{\Nup}$ is unramified at height one primes; see \cite[Theorem 7.3]{SZ13}.\label{intro-list-1}
    \item (\emph{Base change}) If $(R,\fm)$ is $F$-pure and contains a field $k$ isomorphic to $R/\fm$, then $R\otimes_k k'$ is $F$-pure for any finite extension $k\subseteq k'$; see \cite[Proposition 5.4]{Ma14} (\textit{cf.} \cite{Abe01}).\label{intro-list-2}
    \item (\emph{$F$-anti-nilpotence}) If $(R,\fm)$ is $F$-pure, then for every integer $i\in \N$ and every submodule ${L\subseteq H^i_\fm(R)}$ such that $F(L)\subseteq L$, the induced Frobenius action $F\colon\frac{H^i_\fm(R)}{L}\to \frac{H^i_\fm(R)}{L}$ is injective; see~\cite[Theorem~1.1]{Ma14}.\label{intro-list-3}
    \item (\emph{Bertini}$\mkern2mu$) If $X\subseteq \PP^n_k$ is an $F$-pure variety over an algebraically closed field $k$, then so is $X\cap H$ for a general hyperplane $H\subseteq \PP^n_k$; see \cite[Theorem 6.1]{SZ13}.\label{intro-list-4}
    \item (\emph{Deformation to $F$-injectivity}) A local $\F_p$-algebra $(R,\fm)$ is $F$-injective provided that there exists a non-zero divisor $f\in \fm$ such that $R/fR$ is $F$-pure;\footnote{Note that $R$ itself need not be $F$-pure by \cite{Fed83,Sin99}.} see \cite[Corollary 4.3]{HMS14}.\label{intro-list-5}
\end{enumerate}

The four analogous statements \eqref{intro-list-1}--\eqref{intro-list-4} for $F$-injectivity are false in general; see \cite[Proposition 4.2]{Ene09}, \cite[Appendix]{CGM89}, \cite[Proposition 7.4]{SZ13}, and \cite[Example 2.16]{EH08}, respectively. Furthermore, it is one of the preeminent open questions in $F$-singularity theory whether the conclusion of \eqref{intro-list-5} holds under the weaker hypothesis that $R/fR$ is merely $F$-injective.

Geometrically, this is suggestive of $F$-injectivity being an extremal, limiting class within the $F$-singularity hierarchy, at least in the absence of a Frobenius splitting. On the other hand, the failure of $F$-injective rings to satisfy the analogues of \eqref{intro-list-1}--\eqref{intro-list-4} essentially all stems from a single source: the existence of weakly normal rings whose normalization maps ramify in codimension one; see \cite[Appendix]{CGM89}. It is thus tempting to wonder whether \emph{normality} or \emph{geometric normality} hypotheses rectify any of these shortcomings. Our first contribution that we highlight shows that \eqref{intro-list-2} and \eqref{intro-list-3} are still problematic for geometrically normal $F$-injective rings, even in the Cohen--Macaulay case.

\begin{Theoremx}[= \cref{thm:f-inj-not-f-anti}]\label{thm:thm-A}
    For each prime integer $p>0$, there exists a local domain $(R,\fm,k)$ of characteristic~$p$, of dimension $2$, and essentially of finite type over $k$ such that $R$ is geometrically normal over $k$ $($hence Cohen--Macaulay$)$ and $F$-injective, but not $F$-anti-nilpotent.
\end{Theoremx}
The ring $R$ above arises as the localization $S_\fn$ of a standard graded $\F_p(t)$-algebra $S$ at its homogeneous maximal ideal $\fn$, and the purely inseparable base change $S\otimes_{\F_p(t)} \F_p(\!\sqrt[\raisebox{.5ex}{$\scriptscriptstyle p$}]{t})$ is not $F$-injective. Rings satisfying the conclusion of \eqref{intro-list-3} are known as \emph{$F$-anti-nilpotent}, a quality that in \cite{MQ18} was studied alongside a related notion called \emph{$F$-full} (see \cref{subsec:F-singularities}) for its advantageous role in the deformation problem for $F$-injectivity. It is proven in \emph{op.\ cit.}~that if $R/fR$ is either $F$-anti-nilpotent, or $F$-injective and $F$-full, then $R$ is $F$-injective. We also recall that $F$-fullness is an instance of a characteristic-agnostic property known as \emph{cohomological fullness} developed in \cite{DDM21}; independently, Koll\'{a}r and Kov\'{a}cs referred to rings with this property in \cite{KK18} as having \emph{liftable local cohomology}, also through the lens of studying deformation phenomena. One might then hope that \emph{normal} $F$-injective rings are $F$-full with the aim of resolving the deformation question. We show that this is not the case.

\begin{Theoremx}[= \cref{thm:not-f-full}]\label{thm:thm-B}
    For each prime integer $p>0$, there exists a three-dimensional local domain $(R,\fm,k)$ of characteristic $p$ and essentially of finite type over $k$ such that $R$ is $F$-injective, geometrically normal over $k$, but not $F$-full. In particular, $R$ is not Cohen--Macaulay.
\end{Theoremx}

With \cref{thm:thm-A,thm:thm-B} in mind, the following diagram summarizes the implications, or lack thereof, between the various $F$-singularities in local rings $(R,\fm,k)$ essentially of finite type over $k$. When we say that such a ring is \emph{geometrically $\cP$} (over $k$), we mean that $R\otimes_k L$ is $\cP$ for all finite extensions $k\subseteq L$, including purely inseparable ones.

\begin{figure}[htp!]
    \centering
    {\tiny
    \begin{equation*}
    \begin{tikzcd}[column sep=1.1em,row sep=0.58cm]
        \parbox{2cm}{\centering geometrically normal}\arrow[dddddddd,Rightarrow]\arrow[rrrrrrrr,Rightarrow]&&&&&&&&\text{normal}\arrow[dddddddd,Rightarrow]\\
        &&&&&&&&\\
        &&\parbox{1.5cm}{\centering strongly $F$-regular}\arrow[rrrr,Rightarrow]\arrow[uull,Rightarrow]\arrow[dddd,Rightarrow]&&&&\text{$F$-rational}\arrow[uurr,Rightarrow]\arrow[dddl,Rightarrow,degil,sloped,dashed,"\text{\cite{QGSS24}}"' pos =.55,bend left=12]\arrow[dl,dashed,Rightarrow,degil,end anchor={[yshift=+2ex,xshift=-2.7ex]},sloped,"\text{\cite{QGSS24}}" pos = .65]\arrow[dddd,Rightarrow]\arrow[ddr,Rightarrow,bend left=20]&&\\
        &&&\parbox{1.5cm}{\centering geometrically strongly $F$-regular}\arrow[rr,Rightarrow]\arrow[ul,Leftrightarrow,start anchor={[xshift=+.75ex,yshift=-1.2ex]},end anchor={[xshift=-2ex,yshift=-.75ex]},sloped,"\text{\cite{Abe01}}"]\arrow[dd,Rightarrow]&&\parbox{2cm}{\centering geometrically $F$-rational}\arrow[dd,Rightarrow]&&&\\
        &&&&&&&\text{Cohen--Macaulay}&\\
        &&&\parbox{2cm}{\centering geometrically $F$-pure}\arrow[rr,Rightarrow]\arrow[dl,Leftrightarrow,sloped,start anchor={[yshift=+3.7ex,xshift=+.5ex]},end anchor = {[xshift=-.3ex,yshift=+.5ex]},"\text{\cite{Abe01,Ma14}}"']&&\parbox{2cm}{\centering geometrically $F$-injective}&&&\\
        &&\text{$F$-pure}\arrow[ddll,Rightarrow,sloped,"\text{\cite{SZ13}}"]\arrow[rr,Rightarrow,"\text{\cite{Ma14}}"']&&\text{$F$-anti-nilpotent}\arrow[dr,Rightarrow,bend right=15,sloped,"\text{\cite{MQ18}}"']\arrow[from=uuuurr,bend right = 60, sloped,Rightarrow,crossing over,near end,"\text{\cite{Ma14}}"]\arrow[rr,Rightarrow,bend right=8]&&\text{$F$-injective}\arrow[ll,bend right=8,Rightarrow,sloped,dashed,degil,"\text{Thm.~\ref{thm:thm-A}}" near start,"\text{\cite{EH08}}" near end]\arrow[dl,degil,sloped,dashed,Rightarrow,"\text{\resizebox{!}{0.65em}{Thm.~\ref{thm:thm-B}}}"' pos=0.15,"\text{\resizebox{!}{0.55em}{\cite{EH08}}}"' near end]\arrow[ddrr,Rightarrow,sloped,"\text{\cite{Sch09}}"]\arrow[ul,dashed,degil,Rightarrow,sloped, "\text{\resizebox{!}{0.65em}{Thm.~\ref{thm:f-inj-not-f-anti}}}" pos=0.15,"\text{\resizebox{!}{0.55em}{\cite{EH08}}}" near end]&&\\
        &&&&&\text{$F$-full}\arrow[from=uuurr,bend left=40,Rightarrow,sloped,crossing over,"\text{\cite{MQ18}}"', near end]&&&\\
        \parbox{1.5cm}{\centering geometrically (WN1)}\arrow[rrrrrrrr,Rightarrow]&&&&&&&&\parbox{1cm}{\centering weakly normal}
    \end{tikzcd}
\end{equation*}
}
    \caption{Relationships between $F$-singularities for local rings $(R,\fm,k)$ that are essentially of finite type over an $F$-finite field $k$.}\label{diagram:F-singularities}
\end{figure}
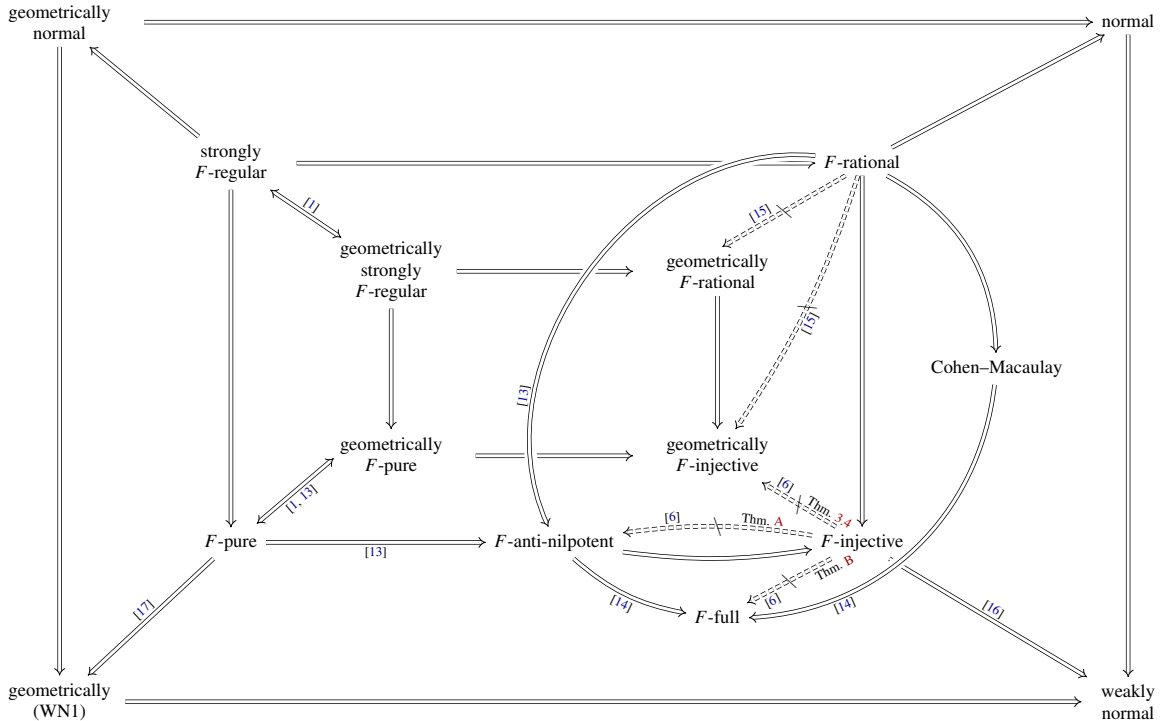

Note that the examples in \cref{thm:thm-A,thm:thm-B} have minimal dimension; this is obvious in \cref{thm:thm-A}. Dimension three is optimal in \cref{thm:thm-B} because normality implies Cohen--Macaulayness in dimension two, hence $F$-fullness; see \cite[Remark 2.4(3)]{MQ18}. We also mention that the examples exhibiting \cref{thm:thm-A} are necessarily not Gorenstein and moreover must have imperfect residue field. Indeed, $F$-purity, $F$-anti-nilpotence, and $F$-injectivity all coincide in the Gorenstein setting by \cite[Lemma 3.3]{Fed83} and \cite[Theorem 3.8]{Ma14}. Also, since $R$ has an isolated singularity, its parameter test ideal $\tpar(R)$ is $\fm$-primary, and \cite[Proposition 4.21]{EH08} dictates that $F$-injective rings with this property are $F$-anti-nilpotent whenever their residue fields are perfect.

We emphasize that the loss of $F$-injectivity along a purely inseparable base change is the force that drives both \cref{thm:thm-A,thm:thm-B}. This pathology resides solely in the singularities on the right side of \cref{diagram:F-singularities}. Enescu's example \cite{Ene09} of an $F$-injective ring that is not geometrically $F$-injective is not normal, whereas the rings found in \cite{QGSS24} are $F$-rational (hence normal) but not geometrically $F$-injective. However, the dimension of the rings found in \cite{QGSS24} grows linearly with the characteristic. Thus, yet another novel feature of the examples of the present article is that they exhibit, for all characteristics, a normal $F$-injective ring of fixed dimension two that is not geometrically $F$-injective, and this dimension is minimal.

\subsection*{Acknowledgments} We thank the referee for a careful reading of the paper and for several useful suggestions.

\section{Preliminaries}
\subsection{Graded ring considerations}
Let $k$ be a field, $R$ a $d$-dimensional $\N$-graded ring finitely generated over the degree zero component $R_0=k$, and $\fm=R_+$ the homogeneous maximal ideal of $R$. Let $x_1,\ldots, x_d$ be a homogeneous system of parameters for $R$. The $\supth{i}$ local cohomology module $H^i_\fm(R)$ supported at $\fm$ is computed as the $\supth{i}$ cohomology of the \v{C}ech complex on $x_1,\ldots, x_d$:
$$\text{\v{C}}(R;x_1,\ldots, x_d)\colon\hspace{1cm} 0\lo R \lo \bigoplus\limits_{i} R_{x_i}\lo \cdots\lo \bigoplus\limits_{i} R_{x_1\cdots \hat{x_i}\cdots x_d}\lo\cdots \lo R_{x_1\cdots x_d}\lo 0.$$
The $R$-module $H^i_\fm(R)$ is then $\Z$-graded (see \cite{GW78}); for $n\in \Z$, we denote by $[H^i_\fm(R)]_n$ its degree $n$ component. In the case of the top cohomology
$$H^d_\fm(R)\cong \frac{R_{x_1\cdots x_d}}{\sum\limits_i \im\left(R_{x_1\cdots \hat{x_i}\cdots x_d}\right)},$$
a \v{C}ech class represented by $\eta=\left[\dfrac{r}{x_1^{n_1}\cdots x_d^{n_d}}\right]$, where $r\in R$ is a homogeneous element, then has degree $$\deg(\eta):=\deg(r) - \sum_i n_i\deg(x_i).$$ If $R$ is Cohen--Macaulay, then the \emph{$a$-invariant of\, $R$} is defined to be
$$a(R):=\max\left\{n\bigm| \left[H^d_\fm(R)\right]_n\neq 0\right\}.$$

We next review the notions of Veronese subrings and Segre products, following \cite{GW78}. Let $$R = \bigoplus\limits_{n\geq 0} R_n, \quad S = \bigoplus\limits_{n\geq 0} S_n$$ be $\N$-graded rings finitely generated over a common field $k = R_0=S_0$, and denote by $\fm_R=R_+$ and $\fm_S = S_+$ the respective homogeneous maximal ideals. Let $d=\dim(R)$ and $d' = \dim(S)$, and suppose from now on that $d,d'\geq 2$. For each positive integer $m\in\N$, the \emph{$\supth{m}$ Veronese subring of\, $R$} is the $\N$-graded ring given by
$$R^{(m)}:=\bigoplus\limits_{n\geq 0} R_{mn}.$$ Fix a positive integer $m\in \N$, and let $\fn$ be the homogeneous maximal ideal of $R^{(m)}$. By \cite[Theorem~3.1.1]{GW78}, 
$$H^i_\fn(T) \cong \bigoplus\limits_{n\in \Z} [H^i_\fm(R)]_{mn}.$$

The \emph{Segre product of\, $R$ and $S$}, denoted $R\#S$, is the $\N$-graded ring given by
$$R\# S:=\bigoplus\limits_{n\geq 0} R_n\otimes_k S_n.$$ Similarly, if $M$ and $N$ are $\Z$-graded $R$- and $S$-modules, respectively, then $$M\# N :=\bigoplus\limits_{n\in \Z} M_n\otimes_k N_n$$ is naturally a $\Z$-graded $R\# S$-module. If $R$ and $S$ have depth at least two, then $\fm_R\#\fm_S$ is the homogeneous maximal ideal of $R\#S$ and the local cohomology modules of $R\#S$ are determined by the  K\"{u}nneth-type formula for local cohomology, see \cite[Theorem 4.1.5]{GW78},
\begin{align}
    H^{\ell}_{\fm_R\#\fm_S}(R\#S) \cong \left(R\# H^\ell_{\fm_S}(S)\right)\oplus \left(H^\ell_{\fm_R}(R)\# S\right) \oplus \bigoplus\limits_{i+j=\ell+1} \left(H^i_{\fm_R}(R) \# H^j_{\fm_S}(S)\right).\label{eq:Kunneth}
\end{align}
We have the pure ring inclusion $R\# S\hookrightarrow R\otimes_{k} S$, the target of which admits an $\N^2$-grading under which the degree $(m_1,m_2)$ component is $R_{m_1}\otimes_{k}S_{m_2}$. The homogeneous maximal ideal of $R\otimes_{k}S$ is the ideal generated by all elements not of degree $(0,0)$. When $R = S$, we refer to $R\otimes_{k} R$ as the \emph{enveloping algebra} of $R$.

\subsection{Hypersurface considerations}\label{subsection-hypersurface} In this paper, we will primarily be concerned with two-dimensional hypersurfaces over a field $k$. Let $S=k[x,y,z]$, $f\in S$, and $R=S/fS$. The Jacobian ideal $\Jac(f)=(f_x,f_y,f_z)$ determines non-smooth points of $R$ over $k$. A prime $\fp$ belongs to $V((f,\Jac(f)))\subseteq \Spec(S)$ if and only if there exists a finite extension $k\to k'$ so that $R_{\fp}\otimes_k k'$ is not regular; see \cite[\href{https://stacks.math.columbia.edu/tag/038W}{Tag 038W}]{stacks-project}. Serre's criterion for normality implies that the hypersurface $R$ is geometrically normal if and only if $S/(f,\Jac(f))$ is supported at finitely many maximal ideals.

Now consider $S$ with an $\N$-grading where the variables $x,y,z$ have positive degree, and let $f$ be a homogeneous polynomial that is monic of degree $m$ with respect to the variable $x$. Then the hypersurface $R = S/fS$ is a free module over the polynomial subring $k[y,z]$ with basis $\{1,x,\ldots, x^{m-1}\}$. Computing the local cohomology module of $R$ with support in the homogeneous maximal ideal $\fm$ using the system of parameters $y,z$ yields a $k$-basis for $H^2_\fm(R)$ of \v{C}ech classes
\begin{align}
    \eta = \left[\frac{x^a}{y^bz^c}\right],\label{eq:hypersurface-basis}
\end{align}
where $0\leq a\leq m-1$ and $b$ and $c$ are positive integers. In particular, the $a$-invariant of $R$ is 
\begin{align}
    a(R)=(m-1)\deg(x)-(\deg(y)+\deg(z)).\label{eq:a-inv}
\end{align}

Further on, we will also require an understanding of the local cohomology of the enveloping algebra $R\otimes_k R$. To that end, let $u,v,w$ denote variables for the second copy of the polynomial ring $S$, let $f'\in k[u,v,w]$ denote the image of $f$ under $x\mapsto u, y\mapsto v,z\mapsto w$, and set $R'=k[u,v,w]/(f')$, with homogeneous maximal ideal $\fm'$. Then $R\otimes_{k}R'\cong k[x,y,z,u,v,w]/(f,f')$ is an $\N^2$-graded ring with homogeneous maximal ideal $\mathfrak{M} = (x,y,z,u,v,w)$ and admits the homogeneous regular sequence $y,z,v,w$. The enveloping algebra $R\otimes_{k}R'$ is then free over $k[y,z,v,w]$ with basis elements $\{z^{\ell_1}w^{\ell_2}\mid 0\leq \ell_1,\ell_2\leq m-1 \}$, and therefore a $k$-basis for $H^4_{\mathfrak{M}}(R\otimes_k R)$ is given by
\begin{align}
\left[\frac{x^{a_1}u^{a_2}}{y^{b_1}z^{c_1}v^{b_2}w^{c_2}}\right],\label{eq:enveloping}
\end{align}
where $0\leq a_1,a_2\leq m-1$ and the $b_i$ and $c_i$ are positive integers. Equivalently, if $\eta = \left[\frac{g}{h}\right]\in H^2_{\fm}(R)$, then let $\eta'\in H^2_{\fm'}(R')$ be the element represented by the \v{C}ech class of the rational function obtained by the change of variables of the rational function $f/g$. If $\mathfrak{M}$ is the $\N^2$-homogeneous maximal ideal of $R\otimes_k R'$ generated by homogeneous elements not of degree $(0,0)$ and if $\Lambda$ is a $k$-basis of $H^2_{\fm}(R)$, then $\{\eta_1\otimes\eta_2'\mid \eta_1,\eta_2\in \Lambda\}$ is a $k$-basis of $H^4_{\mathfrak{M}}(R\otimes_{k} R')\cong H^2_{\fm}(R)\otimes_{k}H^2_{\fm'}(R')$. 

\subsection{\texorpdfstring{$\boldsymbol{F}$}{F}-singularities}\label{subsec:F-singularities}
Let $(R,\fm,k)$ be a local ring of prime characteristic $p>0$ and of Krull dimension $d$. We say that $R$ is \emph{$F$-injective} if the Frobenius maps $F\colon H^i_\fm(R)\to H^i_\fm(R)$ are injective for all $i\leq d$. A locally excellent (not-necessarily local) $\F_p$-algebra $R$ is $F$-injective when $R_\fp$ is so for all $\fp\in\Spec R$. Specializing to our setting, let $(R,\fm,k)$ be an $\N$-graded ring of prime characteristic $p>0$ finitely generated over $R_0=k$. Then $R$ is $F$-injective precisely when the Frobenius map on local cohomology supported at the homogeneous maximal ideal $\fm = R_+$ is injective (see \cite[Theorem 5.10]{DM24}).

In our analysis of when certain Veronese subrings are $F$-injective, we will need to consider hypersurfaces whose local cohomology modules are acted upon injectively by Frobenius in ``sufficiently negative'' degree.

\begin{proposition}[\emph{cf.} {\cite[Remark (iii) on p. 251]{Har95}}]\label{theorem:F-inj-in-negative-degree}
    Let $S = k[x_1,\ldots, x_s]$ be an $\N$-graded polynomial ring over a field $k$ of prime characteristic $p>0$ with $\deg(x_i) = e_i>0$ and homogeneous maximal ideal $\fm$. Let $f\in S$ be a homogeneous polynomial such that $(x_1^{k_1},\ldots, x_s^{k_s})\subseteq \Jac(f)$, where $\Jac(f)$ is the Jacobian ideal of $f$. Let $R = S/fS$. If $n>0$ is such that $pn > \sum_{i=1}^s k_i e_i - 2\deg(f)$, then the Frobenius action $F\colon [H^{s-1}_{\fm_R}(R)]_\ell\to [H^{s-1}_{\fm_R}(R)]_{\ell p}$ is injective for all $\ell\leq -n$.
\end{proposition}
\begin{proof}
  This has already been proven in the case where $k$ is perfect in \cite{Har95}; we explain how to extend this to imperfect fields. Let $L$ be any perfect field extension of $k$ (for example, $L=k^\infty$), and let $T=S \otimes_k L$. It is immediate to check that $\Jac(f)T = \Jac(fT)$, so  all of our assumptions are preserved when extending the scalars to $L$. Let $\fn = \fm T = \fm \otimes_k L$ be the homogeneous maximal ideal of $T$. As $L$ is perfect, we have that the Frobenius action $[H^{s-1}_{\fn}(T/fT)]_{\ell} \to [H^{s-1}_{\fn}(T/fT)]_{\ell p}$  is injective for all $\ell \leq -n$. Since $[H^i_{\fn}(T/fT)]_\ell \cong [H^i_{\fm}(S/fS)]_\ell \otimes_k L$ for all $i,\ell \in \Z$, we conclude that the map $([H^{s-1}_{\fm}(S/fS)]_{\ell} \to [H^{s-1}_{\fm}(S/fS)]_{\ell p}) \otimes_k L$ is injective for all $\ell \leq -n$. It follows that the map $$\left[H^{s-1}_{\fm}(S/fS)\right]_{\ell} \lo \left[H^{s-1}_{\fm}(S/fS)\right]_{\ell p}$$ must be injective as well for all $\ell \leq -n$.
\end{proof}

A local ring $(R,\fm)$ is said to be \emph{$F$-anti-nilpotent} if for all $i\leq \dim(R)$ and all submodules $N\subseteq H^i_\fm(R)$ such that $F(N)\subseteq N$, the induced Frobenius action $$F\colon \frac{H^i_\fm(R)}{N}\lo \frac{H^i_\fm(R)}{N}$$ is injective. From the definition, one sees immediately that $F$-anti-nilpotent rings are $F$-injective. 

We only discuss $F$-anti-nilpotence 
for \emph{local rings}. This is primarily due to the analogue of \cite[Theorem~5.10]{DM24} being unknown for $F$-anti-nilpotence, that is, whether $F$-anti-nilpotence of an $\N$-graded ring may be checked after (non-homogeneous) localization at its homogeneous maximal ideal. Such a result would follow from openness of the $F$-anti-nilpotent locus for an $F$-finite ring; however, at this point, even the more modest question of whether $F$-anti-nilpotence \emph{localizes} remains unanswered.

Descriptions of Frobenius actions on the local cohomology modules of an $\N$-graded ring and its Veronese subalgebras yield a natural method to construct $F$-injective rings that are not $F$-anti-nilpotent. The content of \cref{sec:ex} lies in producing examples of rings that satisfy the assumptions of the following proposition.

\begin{proposition}
    \label{prop: Finding F-inj not anti-nilpotent}
    Let $k$ be a field of prime characteristic $p>0$ and $R$ an $\NN$-graded Noetherian $k$-algebra with homogeneous maximal ideal $\fm$. Assume that $R$ enjoys the following properties:
    \begin{roenum}
    \item There exists an integer $n_0\in\N$ such that for each $i$, the Frobenius action $F\colon[H^i_{\fm}(R)]_{-n}\to [H^i_{\fm}(R)]_{-np}$ is injective if $n\geq n_0$.\label{prop: Finding F-inj not anti-nilpotent-1}
        \item For each $i$, the Frobenius action $F\colon[H^i_{\fm}(R)]_{0}\to [H^i_{\fm}(R)]_{0}$ is injective.
        \item There exist an integer $j\in\N$ and a Frobenius-stable $k$-vector space $V\subseteq [H^j_{\fm}(R)]_{0}$, such that $F(\eta)\in V$ for some $\eta \not \in V$.
    \end{roenum}
    Let $n\geq n_0$ as in \eqref{prop: Finding F-inj not anti-nilpotent-1} be so that $[H^i_{\fm}(R)]_{nq}=0$ for all $i$ and all $q\geq 1$ $($for example, we may choose $n>a(R))$. Let $T:=R^{(n)}$ be the $\supth{n}$ Veronese subalgebra of\, $R$ with homogeneous maximal ideal $\fn$. Then the localization $T_\fn$ is $F$-injective but not $F$-anti-nilpotent.
\end{proposition}

\begin{proof}
    For each $t\in \Z$, we have the identification $[H^i_{\fn}(T)]_{t}=[H^i_{\fm}(R)]_{tn}$. The Frobenius action $[H^i_{\fn}(T)]_{t}\to [H^i_{\fn}(T)]_{pt}$ agrees with the Frobenius action $[H^i_{\fm}(R)]_{tn}\to [H^i_{\fm}(R]_{ptn}$. The assumptions then imply that $T$ (and hence $T_\fn$) is $F$-injective.

    The $k$-subspace $V\subseteq H^j_{\fm}(R)$ is a $k$-subspace of $H^j_{\fn}(T)$ annihilated by the homogeneous maximal ideal $\fn$ of $T$, so $V$ is an $F$-stable $T_{\fn}$-submodule of $H^j_{\fn}(T_\fn)$. Moreover, $\eta\not\in V$ and $F(\eta)\in V$ implies that the induced Frobenius action on $H^j_{\fn}(T_\fn)/V$ is not injective; \textit{i.e.}, $T_\fn$ is not $F$-anti-nilpotent.
\end{proof}

A local ring $(R,\fm)$ of prime characteristic $p>0$ is said to be \emph{$F$-full} provided that the map $\cF_R(H^i_\fm(R))\to H^i_\fm(R)$ is surjective for all $i\geq 0$, where $\cF_R(-)$ denotes the Peskine--Szpiro Frobenius base-change functor. If $R$ is the homomorphic image of a regular local ring of characteristic $p$, say $R\cong S/I$, then $R$ is $F$-full precisely when the natural maps $\Ext^{\dim(S)-i}_S(R,S)\to H^{\dim(S)-i}_I(S)$ are injective for all $i\leq \dim(S)$, which in turn is equivalent to the natural maps $H^i_\fm(S/I^{[p^e]})\to H^i_\fm(R)$ being surjective for all $i\geq 0$ and all $e\geq 0$ (see \cite[Proposition 2.1]{DM24} or \cite[Remark 2.6(1)]{MQ18}). Local rings that are $F$-anti-nilpotent are $F$-full by \cite[Lemma 2.1]{MQ18}. We rely on the following two results to construct $F$-injective rings that are not $F$-full.

\begin{theorem}[\textit{cf.} {\cite[Theorem 4.9]{DDM21}}]\label{F-full}
    Let $(R,\fm)$ be an equidimensional local ring of prime characteristic $p>0$ and Krull dimension $d$. If\, $H^i_\fm(R)$ is a finite dimensional $R/\fm$-vector space for all $i<d$ and $R$ is $F$-full, then the Frobenius actions $F\colon H^i_\fm(R)\to H^i_\fm(R)$ are injective for all $i<d$.
\end{theorem}

\begin{theorem}[\textit{cf.} {\cite[Proposition 3.8]{DDM21}}]\label{F-full-ascent}
    Let $(R,\fm)\to (S,\fn)$ be a flat local ring homomorphism between $F$-finite local rings of prime characteristic $p>0$. If\, $R$ is $F$-full and $S/\fm S$ is Cohen--Macaulay, then $S$ is $F$-full.
\end{theorem}

With the notation of \cref{F-full-ascent}, the strategy for proving \cref{thm:thm-B} is roughly to manufacture $R$ and $S$ to be of dimension three such that $R$ is $F$-injective but such that Frobenius does not act injectively on $H^2_\fn(S)$. The ring $S$ will be non-$F$-full by \cref{F-full}, and hence the same will be true of $R$ by \cref{F-full-ascent}.

\section{The examples}\label{sec:ex}

 We begin by describing some two-dimensional $\N$-graded hypersurfaces over $\F_p(t)$ that will lend themselves to our main theorems. The characteristic $p=2$ and $p>2$ cases are handled separately, giving the following examples.

\begin{example}\label{ex:main-ex}
    Let $p\in \N$ be a prime integer, $t$ be an indeterminate, and $k=\F_p(t)$. Let $S=k[x,y,z]$ be the polynomial ring in three variables over $k$. Let $R=S/fS$, where $f$ is the polynomial specified below.
    \begin{enumerate}
        \item If $p=2$, then let $f = x^3+g\in S$, where $g=tyz^7+y^2z^5+y^3z^3+y^4z+z^9$, homogeneous of degree $9$ with respect to the grading obtained by giving $x,y,z$ degrees $3$, $2$, and $1$ respectively.\label{ex:main-ex-char-2}
        \item If $p>2$, then let $f = x^{p-1}-g\in S$, where $g = ty^{2p-1}+y^{p-1}z^{p^2-p}+z^{2p^2-3p+1}$, homogeneous of degree $2p^2-3p+1$ with respect to the grading obtained by giving $x,y,z$ degrees $2p-1$, $p-1$, and $1$, respectively.\label{ex:main-ex-char-3-or-larger}
    \end{enumerate}
\end{example}

Our first goal in this section will be to show with the aid of \cref{theorem:F-inj-in-negative-degree} that the Frobenius action on $H^2_\fm(R)$ is injective in sufficiently negative degrees. We caution the reader that the hypersurfaces $R$ themselves are not $F$-injective as they have positive $a$-invariant in all characteristics (see \cref{prop:f-inj-in-degree-0-p-equals-2}\eqref{prop:char2-example-a-invariant} and \cref{prop:f-inj-in-degree-0-p-at-least-3}\eqref{prop:char-p>2-example-a-invariant}).

\subsection{A characteristic 2 hypersurface}\label{ex:p-equals-2}
In this subsection, we analyze the hypersurface from \cref{ex:main-ex}\eqref{ex:main-ex-char-2}.

\begin{proposition}\label{prop:f-inj-in-degree-0-p-equals-2}
    Let $k:=\F_2(t)$ and $S:=k[x,y,z]$, where $x,y,z$ have degrees $3$, $2$, and $1$, respectively. Let $f=x^3+tyz^7+y^2z^5+y^3z^3+y^4z+z^9$, a homogeneous element of degree $9$, $R = S/fS$, and $\fm$ be the homogeneous maximal ideal of\, $R$. Then the following statements hold: 
    \begin{enumerate}
        \item The $a$-invariant of\, $R$ is $3$.\label{prop:char2-example-a-invariant}
        \item $R$ is a geometrically normal isolated singularity.\label{prop:char2-example-geometrically-normal}
        \item The Frobenius action $$F\colon\left[H^2_\fm(R)\right]_{-n}\lo \left[H^2_\fm(R)\right]_{-2n}$$ is injective for all $n>6$.\label{prop:f-inj-in-negative-degree-p-equals-2}
        \item The Frobenius action is injective on $[H^2_\fm(R)]_0$. Moreover, $[H^2_{\fm}(R)]_0$ admits a $k$-basis $\{\eta_1,\eta_2,\eta_3\}$ with the property that $F(\eta_1)=\eta_2$, $F(\eta_2)=\eta_1$, and $F(\eta_3)=t\eta_1$.\label{prop:char2-example-injective-in-deg-0}
        \item The Frobenius action is not injective on $[H^2_{\fm'}(R\otimes_k k^{1/2})]_0$, where $\fm' = \fm(R\otimes_k k^{1/2})$.\label{prop:char2-example-not-geo-F-inj-deg-0}
    \end{enumerate}
\end{proposition}

\begin{proof}
As in \eqref{eq:a-inv}, the $a$-invariant of $R$ is $a(R) = 6-(2+1) = 3$, and the Jacobian ideal of $f\in S$ is $$\Jac(f) = (x^2,tz^7+y^2z^3,tyz^6+y^2z^4+y^3z^2+y^4+z^8)=:(x^2,h_1,h_2).$$ We check that $x^2, y^7,z^{11}\in\Jac(f)$. Clearly $x^2\in \Jac(f)$. Let $\alpha = t^2+t+1\in k$, $\beta = t^3+t^2+1\in k$, and $\zeta=t+1\in k$. One may verify that in the polynomial ring $S$,
\begin{align*}
    y^7 &= h_1\cdot(z^7+\beta \alpha^{-1} yz^5+\alpha^{-1} y^2 z^3+\alpha^{-1}t\zeta y^3z)+h_2\cdot (tz^6+t\zeta\alpha^{-1}yz^4+y^2z^2+y^3)\in\Jac(f),\\
    z^{11}&=h_1\cdot(\zeta\alpha^{-1} z^{4}+\alpha^{-1}yz^2+\alpha^{-1}y^2)+h_2\cdot(\alpha^{-1}z^3)\in\Jac(f).
\end{align*}
    
    Because $x^2,y^7,z^{11}\in\Jac(f)$, the maximal ideal $(x,y,z)$ is the only possible non-smooth point over $k$, and~$R$ is geometrically normal. By \cref{theorem:F-inj-in-negative-degree}, the Frobenius action $$F\colon \left[H^2_\fm(R)\right]_{-n}\lo \left[H^2_\fm(R)\right]_{-2n}$$ is injective if $2n>2(3)+7(2)+11(1)-2(9)=13$, \textit{i.e.}, $n>6$. This completes the proof of \eqref{prop:char2-example-a-invariant}--\eqref{prop:f-inj-in-negative-degree-p-equals-2}.

    As in \eqref{eq:hypersurface-basis}, a $k$-basis for the degree zero piece of $H^2_\fm(R)$ is given by
    \begin{equation*}
      [H^2_\fm(R)]_0=\left\langle \left[\frac{x}{yz}\right], \left[\frac{x^2}{y^2z^2}\right], \left[\frac{x^2}{yz^4}\right] \right\rangle =: \langle \eta_1,\eta_2,\eta_3\rangle.
      \end{equation*}

    Observe that $F(\eta_1) = \left[\frac{x^2}{y^2z^2}\right] = \eta_2$. Next, observe that 
    \[
    F(\eta_2) = \left[\frac{x^4}{y^4z^4}\right] = 
    \left[\frac{x(tyz^7+y^2z^5+y^3z^3+y^4z+z^9)}{y^4z^4}\right]=\left[\frac{txz^3}{y^3}+\frac{xz}{y^2} +\frac{x}{yz}+\frac{x}{z^3}+\frac{xz^5}{y^4}\right].
    \]
    The rational functions $\frac{txz^3}{y^3}, \frac{xz}{y^2}, \frac{x}{z^3}$, and $\frac{xz^5}{y^4}$ belong to $\im(R_{yz}\oplus R_{xz}\oplus R_{xy}\to R_{xyz})$. Therefore, 
    \[
    F(\eta_2) = \left[\frac{x}{yz}\right]=\eta_1.
    \]
    The image of $\eta_3$ under the Frobenius action $F$ is similarly found to be
    \[
    F(\eta_3) = \left[\frac{x^4}{y^2z^8}\right] = 
    \left[\frac{tx}{yz}+\frac{x}{z^3} + \frac{xy}{z^5}  + \frac{xy^2}{z^7} + \frac{xz}{y^2}\right]=\left[\frac{tx}{yz}\right] = t\eta_1.
    \]
    Suppose that $\eta = c_1\eta_1 + c_2\eta_2 + c_3\eta_3\in [H^2_{\fm}(R)]_0$ is mapped to zero under the Frobenius action. Then 
    \[
    F(\eta) = c_1^2\eta_2 + c_2^2\eta_1 + c_3^2t\eta_1=(c_2^2+tc_3^2)\eta_1 + c_1^2\eta_2=0.
    \]
    Therefore, $c_2^2+tc_3^2=c_1^2=0$. Since $\{1,t\}$ forms a basis of $k=\F_2(t)$ over $k^2=\F_2(t^2)$, we find that $c_1=c_2=c_3=0$, so $F$ acts injectively on the degree zero piece of $H^2_{\fm}(R)$. On the other hand, the Frobenius action on the degree zero piece of $H^2_{\fm}(R)\otimes_{k}k^{1/2}\cong H^2_{\fm'}(R\otimes_{k}k^{1/2})$ is no longer injective since
    \[
    F(t^{1/2}\eta_2+\eta_3)=t\eta_1+t\eta_1=0.\qedhere
    \]
\end{proof}

\subsection{Hypersurfaces of characteristic \texorpdfstring{$\boldsymbol{p>2}$}{p>2}}\label{ex:p-at-least-3}
In this subsection, we analyze the hypersurface from \cref{ex:main-ex}\eqref{ex:main-ex-char-3-or-larger}.

\begin{theorem}\label{prop:f-inj-in-degree-0-p-at-least-3}
    Let $p>2$ be a prime number, $k=\F_p(t)$, and $S=k[x,y,z]$, where $x,y,z$ have degrees $2p-1$, $p-1$, and $1$, respectively. Let $f=x^{p-1}-(ty^{2p-1}+y^{p-1}z^{p^2-p}+z^{2p^2-3p+1})$, a homogeneous element of degree $2p^2-3p+1$, $R=S/fS$, and $\fm$ be the homogeneous maximal ideal of\, $R$. Then the following statements hold: 
    \begin{enumerate}
        \item The $a$-invariant of\, $R$ is $2p^2-6p+2$.\label{prop:char-p>2-example-a-invariant}
        \item $R$ is a geometrically normal isolated singularity.\label{prop:char-p>2-example-geometrically-normal}
        \item The Frobenius action $$F\colon[H^2_\fm(R)]_{-n}\lo [H^2_\fm(R)]_{-2n}$$ is injective for all $n>3p-7$.\label{prop:char-p>2-example-inj-neg-degree}
        \item The Frobenius action is injective on $[H^2_\fm(R)]_0$. Moreover, there exist linearly independent elements $\eta_1,\eta_2\in [H^2_{\fm}(R)]_0$ such that $F(\eta_1)=\eta_1$ and $F(\eta_2)=t\eta_1$.\label{prop:char-p>2-example-injective-in-deg-0}
        \item The Frobenius action is not injective on $[H^2_{\fm'}(R\otimes_k k^{1/p})]_0$, where $\fm' = \fm(R\otimes_k k^{1/p})$.\label{prop:char-p>2-example-not-geo-F-inj-deg-0}
    \end{enumerate}
\end{theorem}

\begin{proof}
    As in \eqref{eq:a-inv}, the $a$-invariant of $R$ is $a(R)=(2p-1)(p-2)-(p-1+1)=2p^2-6p+2$. The Jacobian ideal of $f$ is 
    \[
    \Jac(f)=(x^{p-2},ty^{2p-2}+y^{p-2}z^{p^2-p},z^{2p^2-3p})S.
    \]
    Observe that
    \[
    ty^{2p-2}z^{p^2-2p} = z^{p^2-2p}(ty^{2p-2}+y^{p-2}z^{p^2-p}) - y^{p-2}(z^{2p^2-3p})\in\Jac(f).
    \]
    Therefore, 
    \[
    y^{3p-2} = t^{-1}y^{p}(ty^{2p-2}+y^{p-2}z^{p^2-p}) - t^{-2}z^p(ty^{2p-2}z^{p^2-2p})\in \Jac(f).
    \]
    By the Jacobian criterion, the maximal ideal $(x,y,z)$ is the only non-smooth point of $R$; therefore, $R$ has a geometrically normal isolated singularity at $\fm$. By \cref{theorem:F-inj-in-negative-degree}, the Frobenius action $F\colon[H^2_{\fm}(R)]_{-n}\to [H^2_{\fm}(R)]_{-pn}$ is injective if
    \begin{align*}
    pn &> (p-2)(2p-1) + (3p-2)(p-1) + (2p^2-3p)(1) - 2(2p^2-3p+1)\\
    & = 3p^2-7p+2.
    \end{align*}
    Equivalently, if $n>3p-7$, then the Frobenius action $F\colon[H^2_{\fm}(R)]_{-n}\to [H^2_{\fm}(R)]_{-pn}$ is injective. This completes the proof of \eqref{prop:char-p>2-example-a-invariant}--\eqref{prop:char-p>2-example-inj-neg-degree}.

    As in \eqref{eq:hypersurface-basis}, a $k$-basis for the degree zero piece of $H^2_\fm(R)$ is given by
\begin{align}
    [H^2_\fm(R)]_0 = \left\langle \eta_{n,m}:=\left[\frac{x^n}{y^m z^{(2p-1)n - (p-1)m}}\right]\;\;\middle|\begin{array}{l}
    1\leq n \leq p-2, \\
    1\leq m\leq 2n
  \end{array}\right\rangle.\label{eq:basis}
\end{align}
Note that $R$ admits a $(\Z/(2p^2-3p+1)\Z)^2$-grading by setting $\deg(x) = (2p-1,0)$, $\deg(y) = (0,p-1)$, and $\deg(z) = (0,1)$. The degree $n$ of the $x$-exponent of a \v{C}ech class from \eqref{eq:basis} is unchanged by the Frobenius action. Indeed, under this multigrading, note that $$\deg(\eta_{n,m}) = ((2p-1)n,-(2p-1)n);$$ hence $\deg(F(\eta_{n,m})) = (p(2p-1)n,-p(2p-1)n)$. Since $\gcd(p,2p-1)=1$, we see that the Frobenius map fixes the value $n$. We therefore fix $1\leq n\leq p-2$, and for each $1\leq m\leq 2n$, we set $\eta_m=\left[\frac{x^n}{y^m z^{(2p-1)n - (p-1)m}}\right]$ as in the basis \eqref{eq:basis}. Suppose there exist $c_1,\ldots, c_{2n}\in k$ such that the element 
    \begin{align}
         c_1\eta_1+\cdots+c_{2n}\eta_{2n} = & \left[\frac{\sum\limits_{m=1}^{2n} c_m y^{2n-m} z^{(p-1)(m-1)}x^n}{y^{2n} z^{(2p-1)n - (p-1)}}\right] \nonumber
    \end{align}
    is mapped to zero under the Frobenius action. That is, 
\begin{align}
\left[\frac{\sum\limits_{m=1}^{2n} c_m^p y^{p(2n-m)} z^{p(p-1)(m-1)}x^{pn}}{y^{2pn} z^{p((2p-1)n - (p-1))}}\right]=0,\label{eq:p-at-least-3-degree-0-1}
\end{align}
which occurs if and only if the element $\sum\limits_{m=1}^{2n} c_m^p y^{p(2n-m)} z^{p(p-1)(m-1)}x^{pn}$ belongs to the ideal
$$I:=(y^{2pn},z^{p((2p-1)n - (p-1))})R.$$ Since $x^{pn} = x^n (x^{p-1})^n = x^n g^n$ in $R$, Equality \eqref{eq:p-at-least-3-degree-0-1} is equivalent to the containment
\begin{align}
    &\left(\sum\limits_{m=1}^{2n} c_m^p y^{p(2n-m)} z^{p(p-1)(m-1)}\right)(ty^{2p-1} + y^{p-1}z^{p^2-p} + z^{2p^2-3p+1})^n\in I.\label{q:p-at-least-3-degree-0-2}
\end{align}
In this case, one may check that for each $1\leq s\leq n$, the coefficient of the monomial $$y^{2pn-s}z^{(p-1)((2p-1)n-p+s)}$$ in the expansion of the left-hand side of \eqref{q:p-at-least-3-degree-0-2} is zero. That is,
\begin{align}
\sum\limits_{i=s}^{2s}\binom{n}{n-s}\binom{s}{i-s} c_i^p t^{i-s} = 0
\end{align}
in $k$. Observe that $p$ does not divide either of the above binomial coefficients since $s\leq n\leq p-2$. Also note  that $c_m^p\in k^p$ for each $m$, and $s\leq p-2\leq p = [k^p(t):k^p]$. This then implies that $c_m = 0$ for all $m$, so $F$ is indeed injective on $[H^2_\fm(R)]_0$. This completes the proof of \eqref{prop:char-p>2-example-injective-in-deg-0}.

Let $\eta_1 = \left[\frac{x}{yz^p}\right]$ and $\eta_2=\left[\frac{x}{y^2z}\right]$. Then $\eta_1,\eta_2\in \left[H^2_{\fm}(R)\right]_0$,
\[
F(\eta_1) = \left[\frac{x^p}{y^pz^{p^2}}\right] = \left[\frac{x(ty^{2p-1}+y^{p-1}z^{p^2-p}+z^{2p^2-3p+1})}{y^pz^{p^2}}\right] = \left[\frac{x}{yz^p}\right] =\eta_1,
\]
and
\[
F(\eta_2) = \left[\frac{x^p}{y^{2p}z^p}\right] = \left[\frac{x(ty^{2p-1}+y^{p-1}z^{p^2-p}+z^{2p^2-3p+1})}{y^{2p}z^{p}}\right] = \left[\frac{tx}{yz^p}\right] =t\eta_1.
\]
We obtain \eqref{prop:char-p>2-example-not-geo-F-inj-deg-0} as a consequence since
\[
F(t^{1/p}\eta_1+(p-1)\eta_2) = pt\eta_1=0
\]
and the Frobenius action on the degree zero piece of $H^2_{\fm}(R)\otimes_kk^{1/p}\cong H^2_{\fm'}(R\otimes_kk^{1/p})$ is not injective.
\end{proof}

\subsection{Proofs of the main theorems}
The following two theorems correspond to \cref{thm:thm-A,thm:thm-B}.

\begin{theorem}\label{thm:f-inj-not-f-anti}
    Let $R$ be the $\N$-graded ring described by \cref{ex:main-ex} with homogeneous maximal ideal $\fm$. If $p=2$, let $n>6$. If $p>2$, let $n>2p^2-6p+2$. Let $T:=R^{(n)}$ be the $\supth{n}$ Veronese subalgebra of\, $R$ and $\fn$ the homogeneous maximal ideal of\, $T$. Then the following statements hold: 
\begin{enumerate}
    \item $T$ $($and hence $T_\fn)$ is $F$-injective and geometrically normal.\label{thm:f-inj-not-f-anti-1}
    \item $T_\fn$ is not $F$-anti-nilpotent.\label{thm:f-inj-not-f-anti-2}
    \item $T\otimes_k k^{1/p}$ is not $F$-injective.\label{thm:f-inj-not-f-anti-3}
    \item The enveloping algebra $T\otimes_k T$ and the Segre product $T\# T$ are not $F$-injective.\label{thm:f-inj-not-f-anti-4}
\end{enumerate}
\end{theorem}

\begin{proof}
The ring $R$ is geometrically normal by \cref{prop:f-inj-in-degree-0-p-equals-2} and \cref{prop:f-inj-in-degree-0-p-at-least-3}. Therefore, the Veronese subalgebra $T$ is geometrically normal by \cite[Proposition 6.15(b)]{HR74}. The ring $T_{\fn}$ is not $F$-anti-nilpotent by \cref{prop: Finding F-inj not anti-nilpotent}. The Frobenius action on $[H^2_{\fm}(R\otimes_{k}k^{1/p})]_{0}$ is not $F$-injective by \cref{prop:f-inj-in-degree-0-p-equals-2} and \cref{prop:f-inj-in-degree-0-p-at-least-3}. The Frobenius action on $[H^2_{\fn}(T\otimes_{k}k^{1/p})]_{0}$ agrees with that of $[H^2_{\fm}(R\otimes_{k}k^{1/p})]_{0}$; therefore, $T\otimes_{k}k^{1/p}$ is not $F$-injective. This completes the proof of \eqref{thm:f-inj-not-f-anti-1}--\eqref{thm:f-inj-not-f-anti-3}.

For \eqref{thm:f-inj-not-f-anti-4}, we adopt notation used in \cref{subsection-hypersurface}. The ring $R$ is presented as $k[x,y,z]/(f)$. Let $f'$ be the polynomial in $k[u,v,w]$ obtained by replacing $x,y,z$ with $u,v,w$, respectively, and set $R'=k[u,v,w]/(f')$ with homogeneous maximal ideal $\fm'$. Consider $R\otimes_k R'$ with its natural $\N^2$-grading, and let $\mathfrak{M}$ be the homogeneous maximal ideal. If $\eta\in H^2_{\fm}(R)$, then $\eta'$ denotes the corresponding element of $H^2_{\fm'}(R')$.

Whether $R$ has characteristic $p=2$ or $p>2$, the vector space $[H^2_{\fm}(T)]_0=[H^2_{\fm}(R)]_0$ admits linearly independent elements $\gamma_1$ and $\gamma_2$ so that $tF(\gamma_1)=F(\gamma_2)$; see \cref{prop:f-inj-in-degree-0-p-equals-2}\eqref{prop:char2-example-injective-in-deg-0} and \cref{prop:f-inj-in-degree-0-p-at-least-3}\eqref{prop:char-p>2-example-injective-in-deg-0}. Then $\gamma_1\otimes \gamma_2'$ and $\gamma_2\otimes \gamma_1'$ are distinct basis elements of 
\[
\left[H^3_{\fn\#\fn}(T\# T)\right]_0\subseteq \left[H^4_{\mathfrak{M}
}(T\otimes_{k} T)\right]_{(0,0)}\subseteq \left[H^4_{\mathfrak{M}}(R\otimes_{k} R)\right]_{(0,0)}
\]
so that
\[
F(\gamma_1\otimes \gamma_2' - \gamma_2\otimes\gamma_1') = F(\gamma_1)\otimes tF(\gamma_1') - tF(\gamma_1)\otimes F(\gamma_1')=t\left(F(\gamma_1)\otimes F(\gamma_1') - F(\gamma_1)\otimes F(\gamma_1')\right)=0.
\]
Therefore, $T\otimes_k T$ and $T\#T$ are not $F$-injective.\footnote{This may also be checked directly in terms of the basis \eqref{eq:enveloping} using $\left[\frac{x^2u^2(z^2v+yw^2)}{y^2z^4v^2w^4}\right]$ when $p=2$ and $\left[\frac{xu(vz^{p-1}-yw^{p-1})}{y^2z^pv^2w^p}\right]$ otherwise.} 
\end{proof}

Finally, we show that the Segre product of the polynomial ring in two variables with the ring constructed above in \cref{thm:f-inj-not-f-anti} produces an $F$-injective, geometrically normal domain that is not $F$-full. Note that such a ring is necessarily non-Cohen--Macaulay by \cite[Remark 2.4(3)]{MQ18}.

\begin{theorem}\label{thm:not-f-full}
    Let $p>0$ be a prime and $R$ the $\N$-graded ring described by \cref{ex:main-ex} with homogeneous maximal ideal $\fm$. If $p=2$, let $n>6$ and $n\equiv 0\pmod{6}$, and if $p>2$, let $n>2p^2-6p+2$ and $n\equiv 0\pmod{(2p-1)(p-1)}$. Let $T:=R^{(n)}$ be the $\supth{n}$ Veronese subalgebra of\, $R$ and $\fn$ the homogeneous maximal ideal of\, $T$. Let $u,v$ be indeterminates of degree $1$ and $A:=T\#k[u,v]$ the Segre product with homogeneous maximal ideal $\fm_A=\fn\#(u,v)$. Then the localization $A_{\fm_A}$ is $F$-injective, geometrically normal, but not $F$-full.
\end{theorem}
\begin{proof}
The Veronese algebra $T=R^{(n)}$ is standard graded since $n$ is chosen to be divisible by the least common multiple of the weights of $x$, $y$, and $z$. The map
\begin{align*}
    A=T\#k[u,v]\longhookrightarrow R\#k[u,v]\longhookrightarrow R\otimes_k k[u,v]\cong \frac{k[x,y,z,u,v]}{(f)}
\end{align*}
is pure since it is a composition of pure homomorphisms, and the target is geometrically normal by \cref{prop:f-inj-in-degree-0-p-equals-2}\eqref{prop:char2-example-geometrically-normal} and \cref{prop:f-inj-in-degree-0-p-at-least-3}\eqref{prop:char-p>2-example-geometrically-normal}. It follows that $A$ is geometrically normal by \cite[Proposition 6.15(b)]{HR74}. Note that $T$ and $T_{\fn}$ are $F$-injective by \cref{thm:f-inj-not-f-anti}\eqref{thm:f-inj-not-f-anti-1}. Therefore, $T_{\fn}\otimes_{k}k[u,v]_{(u,v)}$ and $T\otimes_{k}k[u,v]=T[u,v]$ are $F$-injective by \cite[Theorem A]{DM24}. Also note  that $\dim(A)=3$, and by the K\"{u}nneth formula for local cohomology \eqref{eq:Kunneth}, the only non-zero local cohomology modules of $A$ are
\begin{align}    
    H^3_{\fm_A}(A) \cong & H^2_{\fn}(T) \# H^2_{(u,v)}(k[u,v]),\label{thm:not-f-full-1}\\
    H^2_{\fm_A}(A) \cong & (T\# H^2_{(u,v)}(k[u,v]) \oplus (k[u,v]\# H^2_{\fn}(T))\nonumber\\
        \cong & [H^2_{(u,v)}(k[u,v])]_0 \oplus [H^2_{\fn}(T)]_0\nonumber\\
        \cong & [H^2_{\fn}(T)]_0\nonumber\\
        \cong & [H^2_{\fm}(R)]_0.\label{thm:not-f-full-2}
\end{align}

We see from Equation~\eqref{thm:not-f-full-2}, \cref{prop:f-inj-in-degree-0-p-equals-2}\eqref{prop:char2-example-injective-in-deg-0}, and \cref{prop:f-inj-in-degree-0-p-at-least-3}\eqref{prop:char-p>2-example-injective-in-deg-0} that the Frobenius action on $H^2_{\fm_A}(A)$ is injective. Further, letting $\fM$ be the homogeneous maximal ideal of $T[u,v]$, the Frobenius action on $H^4_\fM(T[u,v])$ is injective by the above discussion and hence restricts to an injective Frobenius action on $H^3_{\fm_A}(A)$ by Equation~\eqref{thm:not-f-full-1} since $H^2_{\fn}(T) \# H^2_{(u,v)}(k[u,v])$ is a graded direct summand of $H^2_\fn(T) \otimes_k H^2_{(u,v)}(k[u,v]) \cong H^4_{\fM}(T[u,v])$. It follows that $A$ and $A_{\fm_A}$ are $F$-injective.

Now let $A' = A\otimes_k k^{1/p}$ and $R' = R\otimes_kk^{1/p}$ with homogeneous maximal ideals $\fm_{A'}$ and $\fm' = \fm R'$. Then the Frobenius map does not act injectively on $H^2_{\fm_{A'}}(A') \cong [H^2_{\fm'}(R')]_0$ by \cref{prop:f-inj-in-degree-0-p-equals-2}\eqref{prop:char2-example-not-geo-F-inj-deg-0} and \cref{prop:f-inj-in-degree-0-p-at-least-3}\eqref{prop:char-p>2-example-not-geo-F-inj-deg-0}. Then  $A'_{\fm_{A'}}$ is not $F$-full by \cref{F-full}, and $A_{\fm_A}$ is not $F$-full by \cref{F-full-ascent}.
\end{proof}


\providecommand{\bysame}{\leavevmode\hbox to3em{\hrulefill}\thinspace}


\begin{thebibliography}{DDSM21+++}

\bibitem[Abe01]{Abe01}
I.\,M.~Aberbach, \emph{Extension of weakly and strongly {F}-regular rings by flat  maps}, J.~Algebra \textbf{241} (2001), no.~2, 799--807,
  \doi{10.1006/jabr.2001.8785}.

\bibitem[CGM89]{CGM89}
C.~Cumino, S.~Greco, and M.~Manaresi, \emph{Hyperplane sections of weakly
  normal varieties in positive characteristic}, Proc.\ Amer.\ Math.\ Soc.\
  \textbf{106} (1989), no.~1, 37--42, \doi{10.2307/2047371}.

\bibitem[DDSM21]{DDM21}
H.~Dao, A.~De~Stefani, and L.~Ma, \emph{Cohomologically full rings}, Int.\ Math.\ Res.\ Not.\ IMRN \textbf{2021}, no.~17, 13508--13545, \doi{10.1093/imrn/rnz203}.

\bibitem[DM24]{DM24}
R.~Datta and T.~Murayama, \emph{Permanence properties of {$F$}-injectivity},
  Math.\ Res.\ Lett.\ \textbf{31} (2024), no.~4, 985--1027,
  \doi{10.4310/mrl.241118233550}.

\bibitem[Ene09]{Ene09}
F.~Enescu, \emph{Local cohomology and {F}-stability}, J.~Algebra \textbf{322}
  (2009), no.~9, 3063--3077, \doi{10.1016/j.jalgebra.2009.04.025}.

\bibitem[EH08]{EH08}
F.~Enescu and M.~Hochster, \emph{The {F}robenius structure of local
  cohomology}, Algebra Number Theory \textbf{2} (2008), no.~7, 721--754,
  \doi{10.2140/ant.2008.2.721}.

\bibitem[Fed83]{Fed83}
R.~Fedder, \emph{{$F$}-purity and rational singularity}, Trans.\ Amer.\ Math.\
  Soc.\ \textbf{278} (1983), no.~2, 461--480, \doi{10.2307/1999165}.

\bibitem[GW78]{GW78}
S.~Goto and K.~Watanabe, \emph{On graded rings. {I}}, J.~Math.\ Soc.\ Japan
  \textbf{30} (1978), no.~2, 179--213, \doi{10.2969/jmsj/03020179}.

\bibitem[Har95]{Har95}
N.~Hara, \emph{{$F$}-injectivity in negative degree and tight closure in graded
  complete intersection rings}, C.~R.\ Math.\ Rep.\ Acad.\ Sci.\ Canada \textbf{17} (1995), no.~6, 247--252.

\bibitem[HR74]{HR74}
M.~Hochster and J.\,L.~Roberts, \emph{Rings of invariants of reductive groups
  acting on regular rings are {C}ohen--{M}acaulay}, Advances in Math.\
  \textbf{13} (1974), 115--175, \doi{10.1016/0001-8708(74)90067-X}.

\bibitem[HMS14]{HMS14}
J.~Horiuchi, L.\,E.~Miller, and K.~Shimomoto, \emph{Deformation of
  {$F$}-injectivity and local cohomology} (with an appendix by K.~Schwede and A.\,K.~Singh), Indiana Univ.\ Math.~J.\ \textbf{63}
  (2014), no.~4, 1139--1157, \doi{10.1512/iumj.2014.63.5313}.

\bibitem[KK20]{KK18}
J.~Koll\'ar and S.\,J.~Kov\'acs, \emph{Deformations of log canonical and
  {$F$}-pure singularities}, Algebr.\ Geom.\ \textbf{7} (2020), no.~6, 758--780,
  \doi{10.14231/ag-2020-027}.

\bibitem[Ma14]{Ma14}
L.~Ma, \emph{Finiteness properties of local cohomology for {$F$}-pure local
  rings}, Int.\ Math.\ Res.\ Not.\ IMRN \textbf{2014}, no.~20, 5489--5509,
  \doi{10.1093/imrn/rnt130}.

\bibitem[MQ18]{MQ18}
L.~Ma and P.\,H.~Quy, \emph{Frobenius actions on local cohomology modules and
deformation}, Nagoya Math.~J.\ \textbf{232} (2018), 55--75,
  \doi{10.1017/nmj.2017.20}.

\bibitem[QGSS24]{QGSS24}
E.~Quinlan-Gallego, A.~Simpson, and A.\,K.~Singh, \emph{Flat morphisms with
  regular fibers do not preserve {$F$}-rationality}, Rev.\ Mat.\ Iberoam.\
  \textbf{40} (2024), no.~5, 1989--2001, \doi{10.4171/rmi/1497}.

\bibitem[Sch09]{Sch09}
K.~Schwede, \emph{{$F$}-injective singularities are {D}u {B}ois}, Amer.~J.\
  Math.\ \textbf{131} (2009), no.~2, 445--473, \doi{10.1353/ajm.0.0049}.

\bibitem[SZ13]{SZ13}
K.~Schwede and W.~Zhang, \emph{Bertini theorems for {$F$}-singularities}, Proc.\
  Lond.\ Math.\ Soc.~(3) \textbf{107} (2013), no.~4, 851--874,
  \doi{10.1112/plms/pdt007}.

\bibitem[Sin99]{Sin99}
A.\,K.~Singh, \emph{Deformation of {$F$}-purity and {$F$}-regularity}, J.~Pure
  Appl.\ Algebra \textbf{140} (1999), no.~2, 137--148,
  \doi{10.1016/S0022-4049(98)00014-0}.

\bibitem[{Sta}25]{stacks-project}
The {Stacks Project Authors}, \emph{The Stacks Project},
  \url{https://stacks.math.columbia.edu}, 2025.

\end{thebibliography}
\end{document}